\documentclass[12pt,oneside]{amsart}
\usepackage{amscd,amsmath,amssymb,amsfonts}
\usepackage[all]{xy}
\usepackage{hyperref}
\usepackage{url}
\usepackage{stmaryrd}
\usepackage{cancel}
\usepackage{ulem}

\usepackage[a4paper, total={6in, 8in}]{geometry}

\usepackage{xcolor}

\usepackage{upgreek}

\theoremstyle{plain}
\newtheorem{thm}{Theorem}
\newtheorem{lem}[thm]{Lemma}

\newtheorem{prop}[thm]{Proposition}

\theoremstyle{definition}

\newtheorem{claim}[thm]{Claim}
\newtheorem{rmk}[thm]{Remark}

\numberwithin{thm}{section}
\numberwithin{equation}{section}

\newcommand{\ga}[2]{\begin{gather}\label{#1}#2 \end{gather}}

\DeclareMathOperator{\im}{{\rm im}}

\newcommand{\C}{{\mathbb C}}

\newcommand{\G}{{\mathbb G}}

\begin{document}

\title{There is no definable Grauert Direct Image Theorem}
\author{H\'el\`ene Esnault \and Moritz Kerz}
\address{Freie Universit\"at Berlin, Berlin,  Germany
}
\email{esnault@math.fu-berlin.de }
\address{   Fakult\"at f\"ur Mathematik \\
Universit\"at Regensburg \\
93040 Regensburg, Germany}
\email{moritz.kerz@mathematik.uni-regensburg.de}
\thanks{ The second author was supported by the SFB 1085 Higher Invariants, Universit\"at Regensburg}

\begin{abstract}
  We prove the claim in the title by showing that a definable Grauert Direct Image
  Theorem in o-minimal geometry would imply a weak
  representability-like property of the definable Picard functor. However, this weak representability cannot hold because of the definable Chow Theorem of Peterzil and Starchenko.
\end{abstract}

\maketitle

\medskip

\section{Introduction}\label{sec:intro}

Definable complex analytic geometry was initiated by Peterzil and Starchenko, see for
example~\cite{PS09}, who obtained fundamental results such as a definable analytic Noether normalization theorem and
the famous definable Chow Theorem.

This foundational work on definable complex analytic geometry was continued by Bakker,
Brunebarbe and Tsimerman in  \cite[Sec.~2]{BBT23}. They applied it to prove Griffiths'
conjecture.  In doing so, they define the
notion of a definable complex analytic space. They then show that many of the basic
results of complex analytic geometry hold in this context, in particular they prove
the  definable  Oka Coherence Theorem \cite[Thm.~2.21]{BBT23} and the definable Grauert Direct Image
Theorem for finite morphisms~\cite[Prop.~2.52]{BBT23}.

It was anticipated by experts that a general definable Grauert Direct Image Theorem might
hold for arbitrary projective maps between algebraic varieties. We had hoped to be able to
apply such a theorem in order to study moduli spaces of definable analytic vector bundles
on compact Riemann surfaces and similar moduli spaces along the lines
of~\cite[p.112]{Sim94}. Unfortunately, it turns out that there can be no definable Grauert
Direct Image Theorem for non-finite morphisms in any reasonable generality. This is the
topic of this note.

Fix any o-minimal structure which expands $\mathbb
  R_{\rm an}$.

\begin{thm}\label{thm:main}
  There exists a smooth projective  definable morphism $f\colon X\to S$ of definable complex analytic spaces
  and a definable coherent sheaf $\mathcal F$ on $X$ such that $f_* \mathcal F$ is not
  definable coherent. More precisely, for $S=\mathbb A^1_\C\setminus \{ 0 \}$,
  $X=Y\times S$ and $f$ the projection onto $S$, where $Y$ is a smooth projective algebraic variety with $H_1(Y,\mathbb Q)\ne 0 $, there always
  exists a semialgebraically definable analytic line bundle $\mathcal L$ on $X$  such that the
  sheaves $f_* \mathcal L(n)$ for $n>0$ cannot all be
  definable coherent with respect to our fixed o-minimal structure.
\end{thm}

In the theorem $f_*$ denotes the pushforward of sheaves on the definable sites relative to the fixed o-minimal structure.

\begin{rmk}
Theorem~\ref{thm:main} may well hold for all o-minimal structures. The assumption that our o-minimal structure expands $\mathbb R_{\rm an}$ is used in an essential way only in Lemma~\ref{lem:gagash}.

\end{rmk}

The idea to prove Theorem~\ref{thm:main}  is to assume that the definable Grauert Direct
Image Theorem holds
and then to deduce a contradiction along the following lines.
Consider a
projective flat morphism
$f\colon X\to S$   with geometrically integral fibers where $X$ and $S$ are complex algebraic
varieties. For a definable line bundle $\mathcal L$ on $X$,   the induced complex analytic
map   $\sigma_{\mathcal L}\colon S\to \mathrm{Pic}(X/S)$    to the Picard scheme would be
definable if a definable Grauert Direct Image Theorem  held, see
Proposition~\ref{prop:defsec}.

 However, it is not difficult to find  such a geometric setting where $S$ is a rational
variety, where $X=Y\times S$ with $Y$ a  smooth projective algebraic variety and where $\mathcal L$ is the line bundle associated to a non-trivial family of rank one
local systems on $Y$ parametrized by $ S$, see Subsection~\ref{subsec:fam}. Then

\[
  \sigma_{\mathcal L}\colon S\to \mathrm{Pic}(Y)\times S \cong \mathrm{Pic}(X/S)
\]
gives rise to a non-constant definable complex analytic morphism from $S$ to the abelian variety
$\mathrm{Pic}^0(Y)$. By the definable Chow Theorem~\cite[Cor.~4.5]{PS09}  this morphism is algebraic, but there is no non-constant
algebraic morphism from a rational variety   $Z$ to an abelian variety $A$.

\medskip

{\it Acknowledgement.} We thank Ben Bakker and Jacob Tsimerman for friendly comments on
our construction and Bruno Klingler for a discussion.
We thank the referee for encouraging us to expand the exposition on the grounds that ``the contents of this paper are too remarkable to be so shortly developed,'' as well as for helpful remarks that improved the presentation.

\section{Preliminaries}

In this section, we summarize some background material on definable complex analytic
geometry and Simpson's construction of moduli. Readers already familiar with these
concepts may skip this section.

\subsection{Reminder on o-minimal structures and definable coherent sheaves}

In this note we use the theory of definable complex analytic spaces and definable coherent
sheaves developed in~\cite[Sec.~2--3]{BBT23}, working over a fixed o-minimal structure
containing $\mathbb R_{\rm an}$. Recall that $\mathbb R_{\rm alg}$ is the o-minimal
structure of real semialgebraic sets (generated by real polynomial equalities and
inequalities), and $\mathbb R_{\rm an}$ is the o-minimal structure generated by
$\mathbb R_{\rm alg}$ together with all restricted real analytic functions, that is,
restrictions to $[-1,1]^n$ of real analytic functions defined on an open neighborhood of
$[-1,1]^n$.

A definable topological space is a topological space $X$ equipped with a finite atlas of
open charts homeomorphic to definable subsets of $\mathbb R^n$ whose transition maps are
definable~\cite[Def.~2.1]{BBT23}.

Associated to $X$ is its definable site $\underline{X}$, whose objects are the definable open subsets of $X$ and whose coverings are finite definable open coverings. Because only finite coverings are allowed, sheaves on $\underline{X}$ behave differently from sheaves on a classical topological space; in particular, the classical points of $X$ do not form a conservative family of points of $\underline{X}$, and exactness of general definable sheaves cannot be checked solely on stalks at classical points~\cite[Rem.~2.8]{BBT23}.

Let $\mathcal O_{\mathbb C^n}$ denote the  presheaf on the definable site $\underline{\mathbb
  C^n}$ assigning to a definable open subset $U\subset \mathbb C^n$ the ring of functions
$U\to \mathbb C$ that are both definable and holomorphic. This presheaf is indeed a sheaf
by~\cite[Lem.~2.17]{BBT23}. Although $\mathcal O_{\mathbb C^n}(U)$ is in general strictly smaller than the ring of all holomorphic functions on $U$ (see~\cite[Example~2.18]{BBT23}), our assumption that the o-minimal structure contains $\mathbb R_{\rm an}$ guarantees that every germ of a holomorphic function at a point $x\in \mathbb C^n$ is definable on a sufficiently small definable neighborhood of $x$; hence the stalk of $\mathcal O_{\mathbb C^n}$ at any classical point coincides with the usual analytic stalk.

A basic definable complex analytic space is a closed definable subspace $X\subset U$ of a
definable open subset $U\subset \mathbb C^n$ cut out by finitely many definable
holomorphic functions $f_1,\dots,f_r\in \mathcal O_{\mathbb C^n}(U)$, equipped with the
definable site $\underline{X}$ and the structure sheaf $\mathcal O_{X^{\rm def}} =
(\mathcal O_{U} / (f_1,\dots,f_r))|_X$. A general definable complex analytic space $X^{\rm
  def}$ is a locally ringed definable space obtained by gluing finitely many basic
definable complex analytic spaces~\cite[Sec.~2.6]{BBT23}.

In particular, every complex algebraic variety $X$  canonically gives rise to a definable
complex analytic space $X^{\rm def}$ as follows: any affine subvariety $V\subset \mathbb
A^n_{\mathbb C}$ is cut out by polynomials and a finite Zariski affine open covering of $X$ has algebraic, hence definable, transition maps, so gluing yields a definable complex analytic space $X^{\rm def}$.

A fundamental result of the theory is the definable analogue of Oka's coherence
theorem~\cite[Thm.~2.21]{BBT23}, which asserts that the structure sheaf $\mathcal
O_{X^{\rm def}}$ is coherent on $\underline{X}$ for any definable analytic space (with
coherence defined in the usual manner via finite generation and finite presentation on
finite definable open coverings). This yields the abelian category ${\rm Coh}(X^{\rm
  def})$ of definable coherent $\mathcal O_{X^{\rm def}}$-modules.

There is a natural morphism of ringed sites $X^{\rm an}\to X^{\rm def}$, whose pullback induces the analytification functor
\[
(-)^{\rm an} \colon {\rm Coh}(X^{\rm def}) \to {\rm Coh}(X^{\rm an}).
\]
Then \cite[Thm.~2.27]{BBT23} shows that $(-)^{\rm an}$ is exact and faithful on ${\rm
  Coh}(X^{\rm def})$.

Moreover, when $X$ is a complex algebraic variety we obtain a definabilization functor $(-)^{\rm def}\colon {\rm Coh}(X)\to {\rm Coh}(X^{\rm def})$ compatible with classical analytification; by o-minimal GAGA~\cite[Thm.~3.1]{BBT23}, this functor is exact and fully faithful, and its essential image is closed under subobjects and quotients.

In this note we primarily focus on definable line bundles on $X^{\rm def}$, that is, locally free $\mathcal O_{X^{\rm def}}$-modules of rank one on the definable site (meaning they admit trivializations over a finite definable open covering of $X^{\rm def}$), and their behavior under the analytification functor $(-)^{\rm an}$.

  \subsection{Simpson's approach to the  Picard scheme}\label{subsec:simpic}

  Let $f \colon X \to S$ be a flat projective morphism of complex algebraic varieties with
  geometrically integral fibers, equipped with a relatively very ample line bundle
  $\mathcal{O}_X(1)$. Fix a Hilbert polynomial $P$. The \'{e}tale sheafification
  $\mathcal P\mathrm{ic}^P_{X/S}$ of the relative Picard functor, which locally on  $T \to S$ parameterizes
  isomorphism classes of line bundles on $X \times_S T$ with fiberwise Hilbert polynomial
  $P$,  is represented by a quasi-projective $S$-scheme
  $\mathrm{Pic}^P({X/S})$.

For a complex analytic space $T$ over $S^{\mathrm{an}}$, let $f_T \colon X_T :=
X^{\mathrm{an}} \times_{S^{\mathrm{an}}} T \to T$ be the base change. Let
$\mathcal{P}\mathrm{ic}^{P,\mathrm{an}}_{X/S}$ denote the analytic sheafification of the
functor assigning to $T$ the set of isomorphism classes of analytic line bundles
$\mathcal{L}$ on $X_T$ with fiberwise Hilbert polynomial $P$.
In~\cite[Sec.~5]{Sim94} Simpson shows (in a much more general context) that
\[
(\mathrm{Pic}^P({X/S}))^{\mathrm{an}} \text{ represents }
\mathcal P\mathrm{ic}^{P,\mathrm{an}}_{X/S} .
\]
 In particular, every analytic line bundle $\mathcal{L}$ on $X_T$ with fiberwise Hilbert polynomial $P$ induces a canonical complex analytic classifying morphism
\[
\sigma_{\mathcal{L}} \colon T \longrightarrow (\mathrm{Pic}^P({X/S}))^{\mathrm{an}}.
\]
In Subsection~\ref{subsec:defpic} we will follow his construction in order to basically show that
the same would hold in the definable analytic world if the definable Grauert Direct Image
Theorem were true.

\smallskip

We now sketch Simpson's approach to the Picard scheme (as a special case of his approach to
moduli of coherent sheaves), which is essential for understanding its analytification.

 By boundedness of line bundles with fixed Hilbert polynomial, locally on $S$ there exists an integer $n \gg 0$ such that for every fiber $X_s$ and every line bundle $\mathcal{L}_s$ on $X_s$ with Hilbert polynomial $P$, the twist $\mathcal{L}_s(n)$ is globally generated and satisfies $H^i(X_s, \mathcal{L}_s(n)) = 0$ for all $i > 0$; in particular, $h^0(X_s, \mathcal{L}_s(n)) = P(n) =: b$.

 Let $ \mathrm{Quot}_{X/S}(\mathcal{O}_X^{\oplus b}(-n), P_n)$ be the relative
 Grothendieck Quot scheme parameterizing quotients of $\mathcal{O}_X^{\oplus b}$ with
 Hilbert polynomial $P$. Inside it, consider the open subscheme $ \mathrm{Quot}' $ parameterizing surjections $\mathcal{O}(-n)_{X_s}^{\oplus b} \to
 \mathcal{L}$ such that $\mathcal{L}$ is a line bundle, $H^i(X_s, \mathcal{L}) =
 0$ for all $i > 0$, and the induced map on global sections $\mathbb{C}^b \to H^0(X_s,
 \mathcal{L})$ is an isomorphism.

The group $\mathrm{GL}_b$ acts on $ \mathrm{Quot}' $ by changing the basis of $\mathcal{O}_X^{\oplus b}$. Because the fibers of $f$ are geometrically integral, $\mathrm{End}( X_s, \mathcal{L}) = \mathbb{C}$, so  the induced action of $\mathrm{PGL}_b$ on $\mathrm{Quot}'$ is free. The relative Picard scheme is obtained as the geometric quotient
\[
q \colon \mathrm{Quot}' \longrightarrow \mathrm{Quot}'/ \mathrm{PGL}_b =: \mathrm{Pic}^P({X/S}),
\]
which is quasi-projective over $S$.
Finally, we recall how one obtains a morphism
\[
  \sigma_{\mathcal{L}}\colon  T \longrightarrow
  \mathrm{Pic}^P({X/S}) \]
for a  line bundle
$\mathcal{L}$ on $X_T$ with fiberwise Hilbert polynomial $P$. For this choose local trivializations of the locally free sheaf $f_{T,*}\mathcal{L}(n)$ yielding local classifying maps to $\mathrm{Quot}'$ whose compositions with $q$ glue to the canonical morphism $\sigma_{\mathcal{L}}$.

\section{The Construction} \label{sec:const}

This section provides a detailed account of the construction outlined in the discussion
following Theorem~\ref{thm:main}.

\subsection{Definability of sections of the Picard scheme}\label{subsec:defpic}
In the following we fix an o-minimal structure which expands $\mathbb R_{\rm an}$.
Throughout the subsequent discussion, the Picard scheme may be replaced, in greater generality, with the coarse moduli scheme of semistable coherent sheaves.

Consider a  flat projective morphism with geometrically integral fibers
$f\colon X\to S$   and with $X$ and $S$ complex algebraic
varieties. Let $\mathrm{Pic}(X/S)$ be Grothendieck's Picard scheme~\cite{Gro95}, whose
construction along the lines of~\cite{Sim94} is recalled in Subsection~\ref{subsec:simpic}.

It
is well-known that for any analytic open subset $U\subset S$ and an analytic line bundle
$\mathcal L$ on $X_U=f^{-1}(U)$ the induced map on $\C$-points $\sigma_{\mathcal L}\colon U
\to \mathrm{Pic}(X/S)|_U$ is
holomorphic, see~\cite[Sec.~5]{Sim94} and Subsection~\ref{subsec:simpic}.

\begin{prop}\label{prop:defsec}
 Let $U\subset S $ be  definable open and let $\mathcal L$ be a definable line bundle on $X_U$.
If for all integers $n>0$, the sheaf $f_* \mathcal L (n)$ on $U$ is definable
coherent, then  the section $\sigma_{\mathcal L}$ is definable.
\end{prop}

Note that the image of $\sigma_{\mathcal L}$ is contained in an open and closed subscheme of $ \mathrm{Pic}(X/S)$ which is of finite type over $S$, given by fixing finitely many Hilbert polynomials, so the notion of definability
makes sense here.  The idea of the proof of Proposition~\ref{prop:defsec} is simply to
copy \cite[Sec.~5]{Sim94}   as briefly recalled in Subsection~\ref{subsec:simpic},
and add the word definable when needed.

\begin{proof}
Assume that the sheaves  $f_* \mathcal L (n)$ are definable coherent for all $n>0$. The  canonical  map
  \[
    (f_* \mathcal L (n))^{\rm an} \xrightarrow{\sim} f_* \mathcal L^{\rm an} (n)
  \]
is an isomorphism by Lemma~\ref{lem:gagash}.
  By Grauert's Base Change Theorem~\cite[Sec.~10.5.5]{GR84} and Lemma~\ref{lem:servan} there exists
  an integer $n$ such that    $$f^* f_* \mathcal L^{\rm an} (n)\to \mathcal L^{\rm an} (n) $$
is surjective. Note also that $f_* \mathcal L^{\rm an} (n)$ is locally free.
As  $f_* \mathcal L (n)$ is definable coherent, and locally free by
Lemma~\ref{lem:locfree}, after shrinking $U$ there is an
isomorphism  $$ \mathcal O_{U}^{\oplus b}  \xrightarrow\sim f_*
\mathcal L (n)$$
which thus induces a surjection  $$  \mathcal O_{X_U}^{\oplus b}(-n) \to \mathcal L .$$
Let $\mathcal K$ be the kernel of this map. Note that $\mathcal K$ is definable
coherent~\cite[Lem.~2.14]{BBT23}.
Again by Grauert's Base Change Theorem and Lemma~\ref{lem:servan} there exists an integer
$m$ such that
\[
R^if_* \mathcal K^{\rm an} (m) = R^if_* \mathcal O^{\rm an}_{X_U}(m-n) = R^if_* \mathcal L^{\rm an} (m) = 0
\]
for all $i>0$. This implies that the analytification of
\[
0\to f_* \mathcal K (m) \to   ( f_* \mathcal O_{X_U} (m-n))^{\oplus b} \to f_* \mathcal L (m ) \to 0
\]
is an exact sequence of locally free $\mathcal O_U^{\rm an}$-sheaves, so this is a short
exact sequence of definable locally free coherent sheaves itself
by~\cite[Thm.~2.27]{BBT23} and  Lemma~\ref{lem:locfree}.

Let $r$ be the rank of
$ f_* \mathcal L (m )$.
This short exact sequence induces a definable analytic map to a
Grassmannian classifying codimension $r$ locally free subsheaves  \[ \tilde
  \sigma_{\mathcal L}\colon U\to \mathrm{Grass}_{U}(   (f_* \mathcal O_{X_U}(m-n))^{\oplus
    b}, r).\]
By the general construction of moduli spaces of coherent algebraic sheaves,
there is an open subscheme $\mathrm{Quot}'$ of finite type over $S$ of Grothendieck's Quot scheme classifying
coherent quotients of  $   \mathcal O_{X}^{\oplus b}(-n)$. Here
$\mathrm{Quot}'$ classifies line bundle quotients with the same Hilbert polynomial as the
line bundles
$\mathcal L_s$ (assume $S$ is connected for simplicity).

For suitable $m$  the above construction provides a locally closed immersion
\[ \mathrm{Quot}'\hookrightarrow  \mathrm{Grass}_S (  ( f_* \mathcal
  O_X (m-n))^{\oplus b}, r) \]
and $\tilde \sigma_{\mathcal L} $ is actually a definable analytic map to $ \mathrm{Quot}'$.
Consider the canonical scheme morphism  $q\colon \mathrm{Quot}' \to \mathrm{Pic}(X/S)$. As $\sigma_{\mathcal L}=
q\circ \tilde \sigma_{\mathcal L}$ is the composition of definable maps, we finally conclude that $\sigma_{\mathcal L}$ is definable.
\end{proof}

The following lemma is a standard variant of Serre's vanishing theorem. It is a
consequence of a combination of Thm.~III.5.2 and  of   Thm.~III.12.11(b) in~\cite{Har77}.

\begin{lem}\label{lem:servan}
For a bounded family of line bundles $\mathcal L^{(\lambda)}_{s_\lambda}$ on fibers $X_{s_\lambda}$, $\lambda\in \Lambda$,
there exists an $n>0$ such that
$
\mathcal L^{(\lambda)}_{s_\lambda} (n)
$
is generated by global sections and such that $ H^i(X_{s_\lambda}, \mathcal L^{(\lambda)}_{s_\lambda} (n))$
vanishes for all $i>0$.
\end{lem}

The proof of the following Lemma~\ref{lem:gagash}  is the  only place where we use that our
o-minimal structure expands $\mathbb
  R_{\rm an}$. It  seems  possible   that  the lemma is true more generally.

\begin{lem}\label{lem:gagash}
  Let $\mathcal F$ be a definable coherent sheaf on $X$. Then the canonical map
  \[
\psi\colon (f_* \mathcal F)^{\rm an} \xrightarrow\sim f_* (\mathcal F^{\rm an})
\]
is an isomorphism.
\end{lem}

\begin{proof}
  It suffices to show that for any $\C$-point $s\in S$ the induced map $\psi_s$ on stalks
  is an isomorphism. Its injectivity is clear. For surjectivity consider
  $\theta\in (f_* \mathcal F^{\rm an})_s$. Say $\theta$ is induced by
  $\tilde \theta \in \mathcal F^{\rm an}( X_V)$, where $V\subset S$ is a definable analytic open
  subset containing $s$. Let $V'\Subset V$ be a relatively compact definable open subset
  containing $s$. Then by Claim~\ref{claim:overcong} $\tilde \theta|_{X_{V'}}$ is induced by
  an element in $\mathcal F (X_{V'})$, which gives rise to a preimage of $\theta$ in
  $(f_* \mathcal F)_s= (f_* \mathcal F)^{\rm an}_s $.
\end{proof}
\begin{claim}\label{claim:overcong}
For a definable coherent sheaf $\mathcal G$ on a definable analytic space $V$ and for $V'\Subset V$
a relatively compact definable open subset, the image of the restriction map $\mathcal G^{\rm an}(V)\to \mathcal
G^{\rm an}(V') $  has values in $\mathcal G(V')\subset  \mathcal
G^{\rm an}(V')  $.
\end{claim}
\begin{proof}
The conclusion of the claim is local, so we can assume without loss of generality that
there is a surjection $\psi\colon \mathcal O_V^b\to \mathcal G$ and that we have to show
the claim for an element of the form $\psi(\zeta)$. In other words we are reduced to  showing
the claim for  $\mathcal G$ the structure sheaf and for $V\subset \mathbb C^n$  a basic
definable complex analytic space~\cite[Def.~2.19]{BBT23}. In this case the claim is clear
by the definition of $\mathbb R_{\rm an}$.
\end{proof}

\begin{lem}\label{lem:locfree}
If $\mathcal F$ is a definable coherent sheaf on a definable analytic space $V$ such that
$\mathcal F^{\rm an}$ is locally free, then $\mathcal F$ is locally free.
\end{lem}

\begin{proof}
  We find a finite definable open covering $V=\cup_{i\in I} V_i$ together with
   surjective morphisms   $\psi_i\colon
\mathcal O_{V_i}^{  \oplus b }\to\mathcal F|_{V_i}$. We can assume that $\mathcal F|_{V_i}^{\rm an}$ is locally
free of rank $r_i$. For a subset $W\subset \{ 1, \ldots , b\}$ of cardinality $r_i$,   let $V_{i,W}\subset V_i$ be the
 open subset over which   $\psi_i|_{\mathcal{O}_{V_i}^{ \oplus W} } $ is an
 isomorphism. Then  $V_i   = \cup_{W} V_{i,W}$, where $W$ runs over the set of  all
 $r_i$-element subsets of $\{1, \ldots, b\}$, as for given generators of a vector bundle
 there locally is a subset which forms a frame of the vector bundle.

 Now $V_{i,W}$ is the complement of the union of the  supports of $\mathrm{Ker}
  (\psi_i|_{\mathcal{O}_{V_i}^{ \oplus W} } )$ and ${\rm Coker} (\psi_i|_{\mathcal{O}_{V_i}^{\oplus W} } )$.
  By \cite[Lem.~2.14]{BBT23} $V_{i,W}$ is thus definable.
Then $V = \bigcup_{i \in I,\, W} V_{i,W}$ forms a finite definable open cover
 such that each restriction $\mathcal{F}|_{V_{i,W}}$ is free, where $W$ ranges over all $r_i$-element subsets of $\{1, \ldots, b\}$.
\end{proof}

\subsection{The character variety and definable local systems of rank one} \label{subsec:fam}

Let $Y$ be a connected complex projective variety. Since $\pi^{\rm ab}_1(Y)$ is finitely generated, its character variety
\ga{}{ {\rm Char}(Y)={\rm Hom}(\pi^{\rm ab}_1(Y), \G_m) \notag}
over $\mathbb C$ is a diagonalizable linear algebraic group which classifies complex local
systems of rank one on $Y$.

For a definable analytic space $U$ we write
\[
\mathrm M (U ) := \left\{ \!\!
    \begin{tabular}{l}
       definable $f^{-1}\mathcal O_U$-sheaves $\mathbb L$ on $X_U$ such that there \\
       exists a finite definable open covering
$Y=\cup_{i\in I} V_i$ \\ with $\mathbb L|_{V_i\times U}\cong f^{-1}\mathcal O_U|_{V_i\times U}$
    \end{tabular} \!\!
\right\} / \text{iso.}
\]
Here $X_U=Y\times U$ is equipped with its definable site, and $f\colon X_U\to U$ is the projection. Note that $\mathrm M (U)$ is an ad hoc definable variant of the framed moduli of local systems~\cite[Lem.~7.3]{Sim94}, which is sufficient for our purpose.

\begin{prop}\label{prop:charvardef}
The canonical map which associates to an element $[\mathbb L]\in \mathrm M(U)$ the family of complex monodromy representations of $\pi_1^{\rm ab}(Y)$ parametrized by $U$ via the fibers $\mathbb L|_{Y\times \{u\}}\otimes_{\mathcal O_{U,u}}\mathbb C$ ($u\in U$),
\[
\Psi\colon \mathrm M( U )\xrightarrow\sim \mathrm{Hom}(U, {\rm Char}(Y) ),
\]
is bijective. Here the codomain consists of all definable analytic maps.
\end{prop}

\begin{proof}
The idea of the proof is to present both sides in terms of certain \v Cech $1$-cocycles modulo
boundaries, which turn out to be isomorphic presentations.

We start with the \v Cech description of the codomain of $\Psi$.
  Let $\mathcal V= (V_i)_{i\in I}$ be a finite definable open covering of $Y$. By~\cite[Ch.~8]{vdD98} we can find such a covering with all $V_i$ contractible. For each $i,j\in I$, the intersection $V_i\cap V_j$ is definable and hence has finitely many connected components. Thus the group of \v Cech $1$-cochains $\check C^1(\mathcal V, \mathbb C^\times) = \prod_{i,j\in I} \G_m^{\pi_0(V_i\cap V_j)}$ is an algebraic torus, and the subgroup $\check Z^1=\check Z^1(\mathcal V, \mathbb C^\times )$ of \v Cech $1$-cocycles has the structure of a complex diagonalizable linear algebraic group. Similarly, $\check C^0=\check C^0(\mathcal V, \mathbb C^\times )\cong \G_m^I$ has the structure of an algebraic torus, and we obtain an exact sequence of diagonalizable linear algebraic groups
\begin{equation}\label{eq:shdiaggr}
\check C^0 \xrightarrow{d} \check Z^1 \xrightarrow{{\rm cl}} {\rm Char}(Y) \to 1,
\end{equation}
where ${\rm cl}$ associates to a cocycle the monodromy representation of the corresponding
local system. Indeed, subdividing any loop in $Y$ into segments contained in members of
$\mathcal V$ expresses the monodromy along that loop as a Laurent monomial in the cocycle
coordinates, so ${\rm cl}$ is a morphism of algebraic groups; it is surjective on $\mathbb
C$-points,  because any rank-one local system on $Y$ restricts to a trivial local system
on each contractible $V_i$, and $\ker({\rm cl})=\im(d)$ because a $1$-cocycle is exact if
and only if the monodromy of its associated local system is trivial.

\medskip

Now we come to the  \v Cech description of the domain of $\Psi$.
Since definable analytic maps $U\to \G_m$ correspond to units in
\[
 \mathcal
  O_U(U) =  (f^{-1}\mathcal O_U)(C\times U)
\]
for every connected definable open subset $C\subset Y$, we have a canonical identification
\[
  \mathrm{Hom}(U, \check Z^1(\mathcal V, \mathbb C^\times )) = \check Z^1(( V_i\times
  U)_{i\in I }, f^{-1} \mathcal O_U^\times ).
\]
Gluing along the finite definable open covering $(V_i\times U)_{i\in I}$ produces a definable locally free $f^{-1}\mathcal O_U$-sheaf of rank one. The canonical map
\begin{equation}\label{eq:homisodef}
 \mathrm{Hom}(U, \check Z^1 ) / \mathrm{Hom}(U, \check C^0 ) \to \mathrm M (U )
\end{equation}
which associates this locally free $f^{-1}\mathcal O_U$-sheaf to a \v Cech cocycle is
therefore well-defined and injective in view of the standard correspondence between first
\v Cech cohomology and isomorphism classes of locally free sheaves. For surjectivity
in~\eqref{eq:homisodef}, let $[\mathbb L]\in \mathrm M(U)$ be trivial over $W_j\times U$
for a finite definable open covering $(W_j)_{j\in J}$ of $Y$. For each $i\in I$, the
restriction $\mathbb L|_{V_i\times U}$ is classified by a \v Cech $1$-cocycle on the
finite definable open covering $\mathcal W_i = (V_i\cap W_j)_{j\in J}$ with values in the
constant sheaf $\mathcal O_U^\times(U)$. Because $\check H^1(\mathcal W_i, \mathcal
O_U^\times(U))$ injects  into $H^1(V_i, \mathcal O_U^\times(U))$  and as this group
  vanishes by the contractibility of $V_i$, this cocycle is a coboundary, so $\mathbb L|_{V_i\times U}\cong f^{-1}\mathcal O_U|_{V_i\times U}$.

Because~\eqref{eq:shdiaggr} is an exact sequence of smooth algebraic groups over $\mathbb C$, it restricts to a short exact sequence of algebraic tori on the connected components of the neutral element, which therefore splits. Translating the resulting section ${\rm Char}(Y)^\circ \to (\check Z^1)^\circ$ by chosen lifts of coset representatives across the clopen connected components of ${\rm Char}(Y)$ yields a morphism of algebraic varieties ${\rm Char}(Y) \to \check Z^1$ splitting ${\rm cl}$ (not necessarily a group homomorphism), and similarly we obtain a variety section $\im(d) \to \check C^0$. Since $\im(d)=\ker({\rm cl})$ scheme-theoretically in characteristic $0$, post-composing with these algebraic sections shows that the sequence of groups of definable analytic maps
\begin{equation}\label{cech:charYU}
  \mathrm{Hom}(U, \check C^0 )\xrightarrow{d} \mathrm{Hom}(U, \check Z^1) \to \mathrm{Hom}(U, {\rm Char}(Y)) \to 1
\end{equation}
is exact.

Now~\eqref{eq:homisodef} and~\eqref{cech:charYU} provide identical presentations of the
domain and codomain of $\Psi$, which implies that $\Psi$ is bijective.
\end{proof}

\subsection{Proof of Theorem~\ref{thm:main}}

As we assume that $H_1(Y, \mathbb{Z}) \cong \pi_1^{\rm ab}(Y)$ is not a torsion group,
  it follows that $\mathrm{Char}(Y)$ has positive dimension.
 So we find a monomorphism
$\mathbb G_m\to \mathrm{Char}(Y)$ of algebraic groups which gives rise by
Proposition~\ref{prop:charvardef} to a definable locally free $ f^{-1}\mathcal O_{\mathbb G_m}$-sheaf $\mathbb L$ of rank one on $X=Y\times {\mathbb G_m}$. Consider the associated definable line bundle
$\mathcal L = \mathbb L \otimes_{ f^{-1}\mathcal O_{\mathbb G_m}} \mathcal O_X$
and the
associated analytic map to the Picard
scheme $\sigma_{\mathcal L}\colon \mathbb G_m \to \mathrm{Pic}(Y/\mathbb C)$. Note that
the image of
$\sigma_{\mathcal L}$ is contained in the connected component of the neutral element, which is an
abelian variety.

\begin{claim}\label{claim:nonconst}
The
map $\sigma_{\mathcal L}$ is
non-constant.
\end{claim}

\begin{proof}
For $s\in \mathbb G_m$ with $|s|=1$
the local system $\mathbb L_s$ on $Y$ is unitary. Non-isomorphic unitary local systems of
rank one
give rise to non-isomorphic line bundles by Hodge theory. So the images by
$\sigma_{\mathcal L}$ of these infinitely many
unitary $s$ are pairwise different, i.e.\  $\sigma_{\mathcal L}$ is not constant.
\end{proof}

Finally, if the definable sheaves $f_* \mathcal L(n)$ were coherent for all $n> 0$ then
$\sigma_{\mathcal L}$ would be definable by Proposition~\ref{prop:defsec}. But then
$\sigma_{\mathcal L}$ would be algebraic by the definable Chow
theorem~\cite[Cor.~4.5]{PS09}. However, algebraic maps from a rational variety to an abelian variety
are constant, see for example~\cite[Cor.~3.9]{Mil86}, so $\sigma_{\mathcal L}$ is constant,
contradicting Claim~\ref{claim:nonconst}.

\end{document}